\documentclass[12pt,english]{article}

\usepackage[english]{babel}
\usepackage[margin=1in,footskip=0.25in]{geometry}
\usepackage{indentfirst,float}
\usepackage{amsmath,amsthm,amssymb,amsfonts}
\usepackage{lineno}
\usepackage{mathtools}
\usepackage{graphicx, color, array}
\usepackage{algorithm}
\usepackage{algorithmic}
\usepackage[utf8]{inputenc}
\usepackage[english]{babel}

\theoremstyle{plain}
\newtheorem{theorem}{Theorem}%[section]
\newtheorem*{theorem*}{Theorem}
\title{Harmonic Oscillators of Mathematical Biology:\\ Many Faces of a Predator-Prey Model}
\author{Sergiy Koshkin* and Isaiah Meyers\\
\\
*Corresponding author\\
Department of Mathematics and Statistics\\
University of Houston-Downtown\\
One Main Street\\
Houston, TX 77002\\
e-mail: koshkins@uhd.edu}

\date{}

\begin{document}

\maketitle

\begin{abstract}
We show that a number of models in virus dynamics, epidemiology and plant biology can be presented as ``damped" versions of the Lotka-Volterra predator-prey model, by analogy to the damped harmonic oscillator. The analogy deepens with the use of Lyapunov functions, which allow us to characterize their dynamics and even make some estimates.
\bigskip

\bigskip


\end{abstract}

In the early 1920s Lotka studied a system of two non-linear differential equations for oscillating concentrations in a chemical reaction \cite{L}, and in 1926 Volterra used the same system to explain the rise in predatory fish populations in the Adriatic sea during World War I \cite{V}. The Lotka-Volterra predator-prey model, as it came to be known, became a paradigmatic example of oscillatory behavior in biology, just as the harmonic oscillator is in physics. In 1927 Kermack and McKendrick applied a special case of it, now called the SIR model, to the spread of infection during epidemics \cite{KK}. Here the ``predators" represented infectious individuals, and those susceptible to the infection were the ``prey". Already in 1929 Soper needed a variation of it, with the inclusion of natural birth rates of the susceptible population, to model measles epidemics in London \cite{Sop}. Further variations and extensions of the Lotka-Volterra model are encountered in many other biological situations, and are generically called predator-prey models \cite{Ber,BJM,CPPW,GKT,HS}.

In this paper, we will study one model of this class that is particularly closely related to the original, we call it the {\it ``damped" predator-prey model}:
\begin{equation}\label{DampedPP}
\begin{cases}\frac{dx}{dt}=\delta - \alpha x - \beta xy\\
\frac{dy}{dt}=\gamma xy- \sigma y.\end{cases}
\end{equation}
When $\beta=\gamma$ and $\delta=0$ this is the classical Lotka-Volterra model, and when also $\alpha=0$, it becomes the SIR model. In the Soper's model, $\alpha=0$ but $\delta>0$. Although system \eqref{DampedPP} has been extensively studied, its connection to the Lotka-Volterra model seems to have escaped attention. We will be able to characterize its behavior in detail by exploiting just this connection. The system, and its extensions, exhibit interesting dynamics that will allow us to introduce and use advanced modeling concepts and tools, like phase portraits, invariant sets, first integrals, separatrices, trapping regions, omega-limit sets, domains of attraction and bifurcations. But the main tool, exploited throughout the paper, will be Lyapunov functions that allow us to prove global stability of solutions to the system, and even obtain some estimates on them. Lyapunov functions are generalizations of energy functions in physics. We will show that for positive parameters system \eqref{DampedPP} behaves somewhat like a {\it damped} harmonic oscillator, an oscillator whose amplitudes decrease due to loss of energy. 
After a general analysis of the system, and its study by means of Lyapunov functions, we will look closely at some of its many applications. First, we will relate it to the classical Lotka-Volterra model, and trace its bifurcation from damped oscillations to predator extinction. Then we will look at a basic model of virus dynamics \cite{BJM}, where the ``predators" are virus producing cells, and the ``prey" are cells susceptible to the infection, and its modification that led to considering $\beta\neq\gamma$. Finally, in a more unexpected incarnation, reminiscent of Lotka's, we derive \eqref{DampedPP} from a simple model of plant growth dynamics, where $x$ and $y$ are concentrations of a nutrient and a growth hormone, respectively. The model itself is an interesting $3$-dimensional extension of  \eqref{DampedPP}. 

We feel that the ``damped" predator-prey model makes for an excellent guided exploration project in a mathematical modeling or differential equations course. Its analysis helps introduce many techniques that are not typically seen in standard examples, and it can be easily modified to model more complicated behavior where they are indispensable. Some variations, including models with limit cycles \cite{GKT}, $3$-dimensional extensions of the virus dynamics model \cite{BJM}, and plant growth models that make different simplifying assumptions about water transport \cite{BV}, can be used as a basis for student research projects. 

\section{Equilibria and invariant regions}

The first step when dealing with a system of differential equations like \eqref{DampedPP} is to find its equilibria, which are solutions that do not change with time. This means $\frac{dx}{dt}=\frac{dy}{dt}=0$. We find two of them: $\left(\frac{\delta}{\alpha}, 0 \right)$ and $\left(\frac{\sigma}{\gamma},\frac{a}{\beta}\right)$, where $a:=\frac{\gamma\delta}{\sigma}-\alpha$. In addition to assuming that all parameters of the system are positive, we shall also assume for now that $\gamma\delta>\sigma\alpha$. This makes $a>0$ and the both equilibria are in the first quadrant. To get an idea about the overall behavior of the trajectories, i.e. of the {\it flow}, it is instructive to plot the right hand side of \eqref{DampedPP} as vectors attached to points in the $x$-$y$ plane, the \textit{vector field} of the system, Figure \ref{SlopePP}(a). Trajectories must be tangent to the slope field vectors at every point.

\begin{figure}[!ht]
\begin{center}
(a)\ \ \ \includegraphics[scale=0.47]{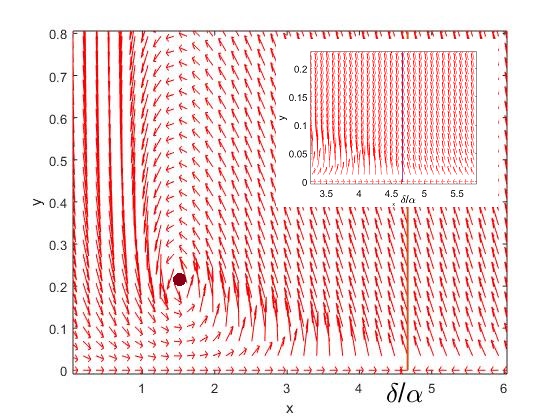} \hspace{0.1in} 
(b)\ \ \ \includegraphics[scale=0.35]{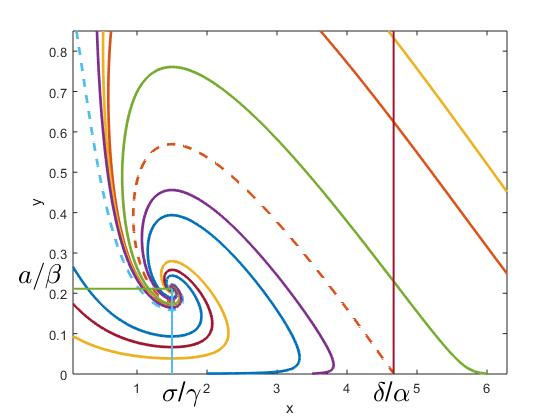}\par
\end{center}
\caption{\label{SlopePP} (a) Vector field of \eqref{DampedPP} for $\delta>0$ with inset figure of its behavior near $(\delta/\alpha,0)$; (b) Flow of \eqref{DampedPP} in the first quadrant. Dashed lines are separatrices.}
\end{figure}
Note that the field restricted to the $x$-axis is $(\delta-\alpha x,0)$, so parallel to it and pointing towards the equilibrium point $\left(\frac{\delta}{\alpha},0\right)$. Along the $y$-axis the field is $(\delta,-\sigma y)$, and is either parallel to it (if $\delta=0$) or points inside the first quadrant (if $\delta>0$). Since trajectories are directed along the field vectors they can not leave the first quadrant if they start in it. In general, regions that trajectories can not leave are called {\it flow-invariant}. A closer look at the line $x=\frac{\delta}{\alpha}$ for $\delta>0$ shows that the field restricted to it, $\left(-\frac{\beta\delta}{\alpha}y,\left(\frac{\gamma\delta}{\alpha}-\sigma\right)\!y\right)$, always points leftward. This means that the half-strip $H:=\{(x,y)\in\mathbb{R}^{2}|0\leq x\leq\frac{\delta}{\alpha},y\geq 0\}$ is also flow-invariant. 

Moreover, for small $y\neq 0$, the vectors along the line $x=\frac{\delta}{\alpha}$ point away from the equilibrium point $\left(\frac{\delta}{\alpha},0\right)$, which means that this equilibrium is \textit{unstable}. Trajectories with $y\neq 0$ will move away from it no matter how small $y$ is. Conversely, the other equilibrium seems to be \textit{stable}, more precisely, locally asymptotically stable, meaning nearby trajectories flow into it. 

This simple analysis allows us to form a preliminary picture of the flow in the first quadrant depicted on Figure \ref{SlopePP}(b). There is a family of trajectories entering $H$ from the left, through the $y$-axis, and another family entering it from the right, through the line $x=\frac{\delta}{\alpha}$. By continuity, there has to be a special trajectory separating these two families. Such trajectories are called \textit{separatrices} and the behavior of the vector field along the $y$-axis suggests that this one should be asymptotic to it. There has to be another separatrix separating trajectories passing under the stable equilibrium and going up from those entering $H$ from the right. This one seems to be ``originating" (at $t\to-\infty$) at the unstable equilibrium. The trajectories entering $H$ from the right are squeezed between these two separatrices.

\section{Harmonic oscillator, damped and undamped}

As compelling as the above picture is, it is only a picture. 
To move past mere illustrations, it will be helpful to look at the physical cousin of our model, the (damped) harmonic oscillator. To save space, from now on we will denote time derivatives by dots, e.g.  $\dot{f}:=\frac{df}{dt}$ and $\ddot{f}:=\frac{d^2f}{dt^2}$. The harmonic oscillator equation $\ddot{y}+by=0$, and its damped version $\ddot{y}+a\dot{y}+by=0$, are commonplace in mathematical physics, describing phenomena as diverse as springs, pendula, RLC electric circuits, tuning forks, atoms in a solid, etc. \cite{Jha}. Setting $x:=\dot{y}$ we present the oscillator as a $2$-dimensional system
\begin{equation}\label{undampedOsc}
\begin{cases}\dot{x}=-ax-by\\
\dot{y}=x\,.\end{cases}
\end{equation}
It is instructive to visualize solutions $(x(t),y(t))$ to this system as moving points in the $x$-$y$ plane called the {\it phase plane} of the system. When $a=0$ equations of the curves they traverse can be found by dividing the second equation by the first, separating variables and integrating: 
\begin{eqnarray}
\nonumber \frac{dy}{dx}&=&\frac{\dot{y}}{\dot{x}}
=-\frac{x}{by}\,.
\end{eqnarray}
The result, called the {\it first integral} of \eqref{undampedOsc} with $a=0$, is
\begin{eqnarray}
\nonumber x^{2}+by^{2}&=&C.
\end{eqnarray}
The curves it defines for $C>0$ are ellipses centered at the origin, and the points move along them counterclockwise. These ellipses are level sets of the function $V(x,y):=x^{2}+by^{2}$, which represents the total energy of the oscillator.
\begin{figure}[!ht]
\begin{center}
(a)\ \ \ \includegraphics[scale=0.35]{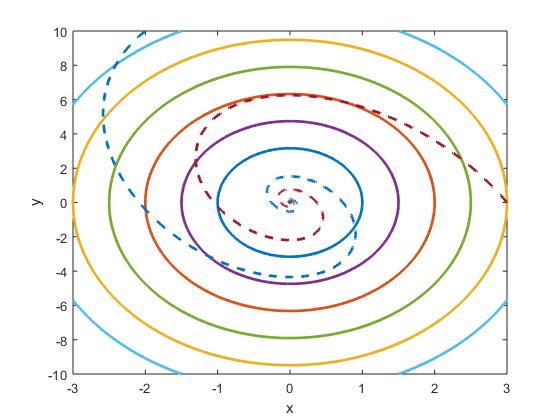} \hspace{0.1in}
(b)\ \ \ \includegraphics[scale=0.35]{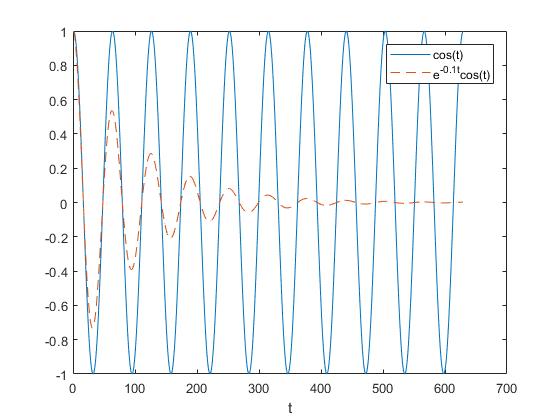}\par
%(b)\par
\end{center}
\caption{\label{HarmOsc} The undamped (solid) and damped (dashed) trajectories of the harmonic oscillator: (a) in the phase plane; (b) as functions of time.}
\end{figure}

When the damping coefficient $a>0$, the oscillator is losing energy (say, to friction), and the phase trajectories spiral into the origin, always staying within every ellipse they enter, Figure \ref{HarmOsc}(a). Computing the derivative of energy $V(x,y)$ along the trajectories of the damped system \eqref{undampedOsc} we find: 
$$
\frac{d}{dt}V(x(t),y(t))=\frac{\partial V}{\partial x}\dot{x}+\frac{\partial V}{\partial y}\dot{y}=-2ax^{2}\leq0\,.
$$
As expected, $V(x(t),y(t))$ is a decreasing function of time this is why the damped trajectories can enter the ellipses but never exit them. This is a prototypical example of a Lyapunov function for systems of differential equations.

\section{Limit points and Lyapunov functions}

To appreciate the utility of Lyapunov functions it is instructive to imagine what can generally happen to trajectories. Going back to our model \eqref{DampedPP}, let us abbreviate $p(t)=(x(t),y(t))$ the solution starting at a point $p_{0}=(x_{0},y_{0})$ in the first quadrant. What our intuitive picture suggests is that $p(t)\to p_{*}:=\left(\frac{\sigma}{\gamma},\frac{a}{\beta}\right)$ when $t\to\infty$. But nothing we said so far precludes $p(t)$ from ``escaping to $\infty$" instead, e.g. moving up $H$ indefinitely. Something yet more curious can happen. Rather than spiraling into $p_{*}$, our $p(t)$ may get stuck cycling around it forever.
\begin{figure}[!ht]
\begin{center}
(a)\ \ \ \includegraphics[scale=0.49]{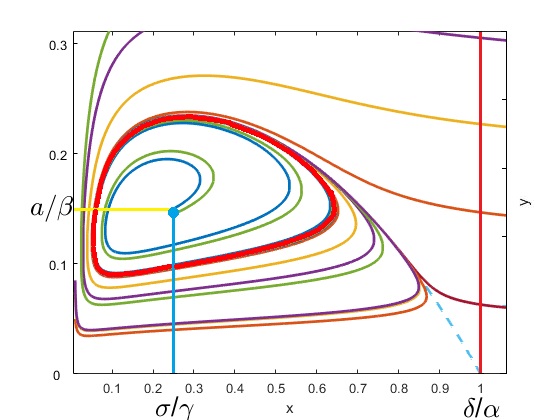} \hspace{0.1in} 
(b)\ \ \ \includegraphics[scale=0.37]{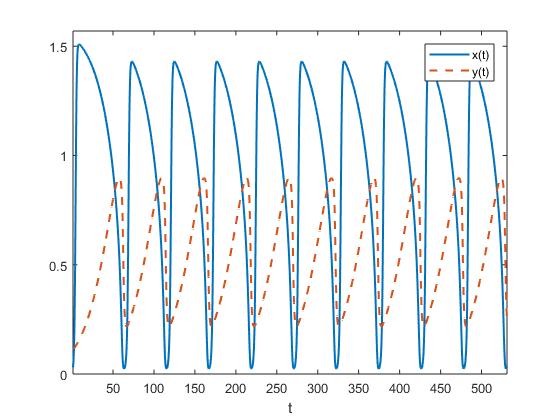}\par
\end{center}
\caption{\label{limitCycle} Holling-Tanner system: (a) limit cycle in the phase plane; (b) phase variables as functions of time.}
\end{figure}
This would mean having yet another separatrix, a closed curve that  separates the inside and the outside trajectories. Such separatrices are called \textit{limit cycles}. Limit cycles do indeed occur in some predator-prey models, like the Holling-Tanner system \cite{GKT}, see Figure \ref{limitCycle}. They correspond to solutions that display a pattern of sustained oscillations.

How do we detect (or rule out) such behavior? This is where Lyapunov functions comes in. With their help we can rule out escapes to $\infty$ and limit cycles all at once. If trajectories escaped to $\infty$ the Lyapunov function would eventually have to grow along them, which it can not. If there was a limit cycle its time derivative would have to vanish on it, and this can be checked.

Asymptotic behavior of trajectories is reflected by their limit points. For a trajectory that starts at $p_{0}$ the set of such points is called the {\it omega-limit set} of $p_{0}$ \cite[2.6]{HS}:
\begin{equation*}
\omega(p_{0}):=\{q\in\mathbb{R}^{2}\,|\,p(t_{k})\to q \textrm{ for some } t_{k}\to\infty\}.
\end{equation*}
It is easy to show that $\omega(p_{0})$ is always closed, i.e. contains its own limit points, and flow-invariant. It is non-empty if $p(t)$ is bounded, and then it is itself bounded. If $p(t)$ escapes to $\infty$ this set will be empty. And if $p(t)$ approaches a limit cycle, $\omega(p_{0})$ will contain the whole  cycle. 

Suppose $p(t)$ stays within some region $U$ for all $t\geq0$, a {\it trapping region}, and let $V(p)$ be a function defined on $U$. Consider the time derivative of $V$ along the trajectories of the system: 
$$
\dot{V}(x,y):=\nabla V\cdot(\dot{x},\dot{y})=\frac{\partial V}{\partial x}\dot{x}+\frac{\partial V}{\partial y}\dot{y}\,.
$$
Then $\frac{d}{dt}V(p(t))=\dot{V}(p(t))$ and if $\dot{V}(p)\leq0$ on $U$ then $V(p(t))$ is non-increasing along the trajectories. A version of the theorem due to Lyapunov \cite[2.6]{HS} tells us more. 
\begin{theorem}[Lyapunov]\label{Lyap}
Let $V$ be a continuously differentiable function defined on a region $U$ such that $\dot{V}(p)\leq0$. If $p(t)$ stays within $U$ for all $t\geq0$ then $\dot{V}$ vanishes on the limit points of $p(t)$ that are within $U$, i.e. $\omega(p_{0})\cap U\subseteq\dot{V}^{-1}(0)$.
\end{theorem}
In the classical case considered by Lyapunov $V$ was required to have a global minimum, a single point where $\dot{V}=0$ which also happens to be an equilibrium of the system. But these conditions are so demanding that not even the damped harmonic oscillator satisfies them all! Indeed, $\dot{V}=-2ax^{2}=0$ on the entire line $x=0$, not just at $(0,0)$. We will call a function $V$ a \textit{$Lyapunov$ $function$} if it simply satisfies $\dot{V}\leq0$, i.e. if it does not increase along the trajectories. Figure \ref{LyapBowl}\,(a) illustrates how Lyapunov functions work. It depicts a trajectory $p(t)$ together with the value of $V(p(t))$, i.e. a trajectory lifted to the graph of $V$. The graph of $V$ is bowl shaped, and the fact that $V$ decreases along the trajectories means that their lifts are funneled towards its bottom, ideally the global minimum.

\begin{figure}[!ht]
\begin{center}
(a)\ \ \ \includegraphics[scale=0.48]{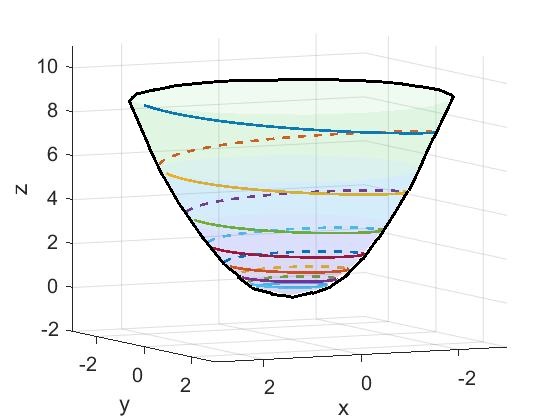} \hspace{0in} 
(b)\ \ \ \includegraphics[scale=0.48]{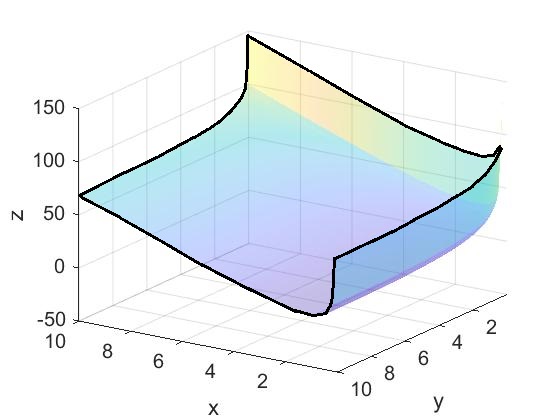}\par
\end{center}
\caption{\label{LyapBowl} (a) Lyapunov lifted trajectory of the damped harmonic oscillator; (b) Lotka-Volterra Lyapunov function.}
\end{figure}

\section{Lyapunov function for the ``damped" predator-prey model}

Let us start with the predator-prey model of Lotka-Volterra, a biological analog of the undamped harmonic oscillator:
\begin{equation}\label{Vsys}
\begin{cases}\dot{x}=ax-\beta xy\\
\dot{y}=\gamma xy -\sigma y.\end{cases}
\end{equation}
Here $x$ is the number of prey fish and $y$ the number of their predator. The prey multiplies at the rate $ax$, proportional to the size of its population, and is consumed at a rate proportional to its encounters with the predator. It is a common assumption in population dynamics that these encounters are random and therefore proportional to the product of the population numbers $xy$. This is called the mass action principle. Conversely, the population of predators multiplies proportionally to $xy$ and has the death rate $\sigma y$.  Unlike \eqref{undampedOsc} this model is nonlinear, but we can still find the first integral by the same trick:
\begin{equation*}
\frac{dy}{dx}=\frac{\gamma xy -\sigma y}{ax-\beta xy}
=\frac{\gamma-\frac{\sigma}{x}}{\frac{a}{y}-\beta}.
\end{equation*}
Separating the variables and integrating, we compute:
\begin{equation}\label{firstInt}
\gamma x-\sigma\ln x+\beta y-a\ln y=C.
\end{equation}
We assume that all the parameters are positive and only look at the positive values of $x$ and $y$ since neither populations nor concentrations can be negative. Then on the left we have a sum of two one-variable functions that, like parabolas, are convex down and have a global minimum. The resulting picture in the $x$-$y$ plane is topologically similar to Figure \ref{HarmOsc}(a), but the ellipses are replaced by ovals confined to the first quadrant and the center is no longer at the origin, see Figure \ref{PredExt}(a). 

For our model \eqref{DampedPP} we shall take Lotka-Volterra's first integral as a Lyapunov function template (such a template with free parameters is often called {\it ansatz}), but without specifying yet what $a$ is:
\begin{equation}\label{Lyapunov1}
V(x,y)=\gamma x - \sigma\ln x + \beta y - a\ln y.
\end{equation}
Then we differentiate,
$$
\nabla V=\left(\gamma-\frac{\sigma}{x},\beta-\frac{a}{y}\right),$$
and split the vector field of \eqref{DampedPP} as follows:
$$
X:=\left(\delta - \alpha x - \beta xy,\gamma xy- \sigma y\right)
=\left(ax - \beta xy,\gamma xy- \sigma y\right)
+\left(\delta - \alpha x - ax,0\right)=:X_0+Z.
$$
The first field $X_0$ is from the Lotka-Volterra system \eqref{Vsys}. We know that $\nabla V\cdot X_0=0$, because $V$ is constant along its trajectories, so $\nabla V\cdot X=\nabla V\cdot Z$, and 
\begin{equation*}
\dot{V}(p)=\nabla V\cdot Z
=\left(\gamma-\frac{\sigma}{x}\right)\left(\delta - (a+\alpha)x\right)
=-\frac{\gamma(a+\alpha)}{x}\left(x-\frac{\sigma}{\gamma}\right)\left(x-\frac{\delta}{a+\alpha}\right).
\end{equation*}
Now it becomes clear how to select $a$, which we left indeterminate, to make $V$ a Lyapunov function. If we set $\frac{\sigma}{\gamma}=\frac{\delta}{a+\alpha}$, and $a>0$, then
\begin{equation}\label{LyapunovDer1}
\dot{V}(x,y)=-\frac{\gamma^2\delta}{\sigma x}\left(x-\frac{\sigma}{\gamma}\right)^2\leq0.
\end{equation}
\begin{figure}[!ht]
\begin{center}
(a)\includegraphics[scale=0.39]{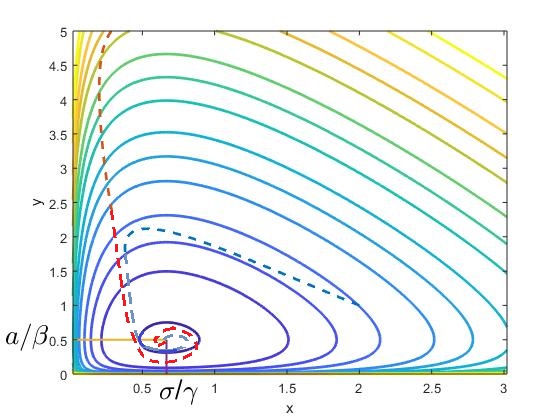}
(b)\ \includegraphics[scale=0.19]{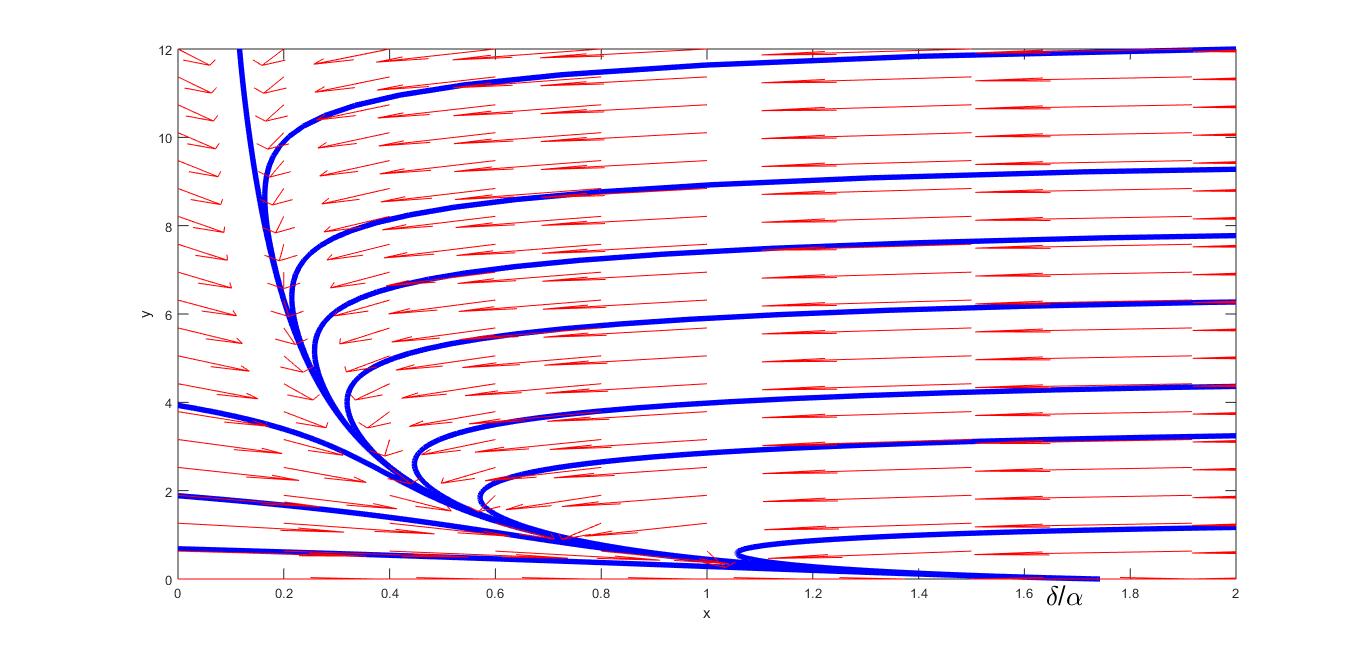}
\caption{\label{PredExt} Flows for (a) the Lotka-Volterra model (solid) and its ``damped" version (dashed) with $a>0$; (b) predator extinction with $a<0$.}
\end{center}
\end{figure}
This means that $a:=\frac{\gamma\delta}{\sigma}-\alpha$, the same $a$ we introduced when computing equilibria. With our additional restriction $\gamma\delta>\sigma\alpha$, we have $a>0$. We will discuss what happens when it is negative shortly. Note that moving from system \eqref{DampedPP} to \eqref{Vsys} does not amount to simply setting a single ``damping" coefficient to $0$, as it was with the harmonic oscillator, the relation between the ``damped" predator-prey system and its ``undamped" version is more complicated. We should also mention that there are other ways of ``damping" the Lotka-Volterra model \cite{Gar}.

\section{Global attraction}

Now that we have a Lyapunov function, to apply Theorem \ref{Lyap} we need to choose a trapping region $U$. It is tempting to take the entire first quadrant as $U$. But this would not work for two reasons. First, we can not include the axes since our $V$ from \eqref{Lyapunov1} is not defined on them. Second, even without the axes $U$ is unbounded, so we can not rule out $\omega(p_{0})$ being empty. A resolution is to take as $U$ the Lotka-Volterra oval from Figure \ref{PredExt}(a) with $p_{0}$ on its boundary, i.e. 
$$
U_{p_{0}}:=\{p\in\mathbb{R}^{2}\,|\,V(p)\leq V(p_{0})\}\,.
$$
This region is closed, bounded, and $p(t)$ stays within it for all $t\geq 0$ since $V(p(t))$ is non-increasing. Now we can confirm our picture of the flow rigorously.
\begin{theorem}\label{GlobAttr}
Suppose $\alpha,\beta,\gamma,\delta,\sigma>0$ and $\gamma\delta>\alpha\sigma$. Let $p_{*}:=(\frac{\sigma}{\gamma},\frac{a}{\beta})$ be the (stable) equilibrium of system \eqref{DampedPP} with $a:=\frac{\gamma}{\sigma}\delta-\alpha$. Then for any $p_{0}=(x_{0},y_{0})$ with $x_{0},y_{0}>0$ the trajectory starting at $p_0$ converges to $p_*$ at $t\to\infty$.
\end{theorem}
\begin{proof}
Any $p_{0}$ belongs to a Lotka-Volterra oval $U_{p_{0}}$. Since $U_{p_{0}}$ is flow-invariant, $p(t)\in U_{p_{0}}$ for all $t\geq 0$. Since it is closed and $\omega(p_{0})$ consists of limit points of $p(t_{k})$ we have $\omega(p_{0})\subseteq U_{p_{0}}$. Hence, by the Lyapunov theorem, we have $\omega(p_{0})\subseteq \dot{V}^{-1}(0)$. But $\dot{V}(p)=-\frac{\gamma^{2}\delta}{\sigma x}\left(x-\frac{\sigma}{\gamma}\right)^{2}$, so $\dot{V}^{-1}(0)$ is the line $x=\frac{\sigma}{\gamma}$. But $\omega(p_{0})$ must be flow-invariant, while trajectories of \eqref{DampedPP} move off this line unless $\dot{x}=0$, i.e. unless also $y=\frac{a}{\beta}$. Therefore, $\omega(p_{0})\subseteq\{p_{*}\}$. Since $U_{p_{0}}$ is bounded $\omega(p_{0})$ is non-empty, so $\omega(p_{0})=\{p_{*}\}$  and $p_{*}$ is the limit of $p(t)$.
\end{proof}
If trajectories originating from points of a certain region converge to an equilibrium, this area is called its {\it domain of attraction}. Our theorem says that the interior of the first quadrant is a domain of attraction of $\left(\frac{\sigma}{\gamma},\frac{a}{\beta}\right)$ when $\gamma\delta>\alpha\sigma$. Biologically, this means that the predator and the prey (or the infectious and the susceptible) populations stabilize at the equilibrium values.

\section{Bifurcation to predator extinction}

As we saw, when $\gamma\delta\leq\alpha\sigma$ function \eqref{Lyapunov1} ceases to be a Lyapunov function of the ``damped" predator-prey system. This is not an artifact of its choice. As $a=\frac{\gamma\delta-\alpha\sigma}{\sigma}-\alpha$ decreases, the stable equilibrium, the center of the ovals in Figure \ref{PredExt}(a), moves to the $x$-axis, and merges with the unstable one when $a=0$. Then it moves below the $x$-axis when $a<0$ and is no longer relevant to the first quadrant. As a result, geometry of the trajectories changes dramatically. According to the vector field, nearby trajectories now seem to approach the formerly unstable equilibrium $(0,\frac{\delta}{\alpha})$. This kind of stability change under variation of parameter values is called \textit{transcritical bifurcation} \cite{SS}. 

Since the $x$-coordinate of our suspected attracting equilibrium is $0$ after the bifurcation we will choose $V(x,y)=\gamma x-b\ln x +\beta y$ as our Lyapunov function ansatz, with $b$ to be determined. This is the first integral of a degenerate Lotka-Volterra system \eqref{Vsys} with $a=0$ and $\sigma=b$. This system is of interest in its own right, it is called the SIR (for Susceptible-Infected-Removed individuals) model in epidemiology \cite{AL}, but we will not dwell on it here.
Performing calculations as in the previous section, we find that $b:=\frac{\delta}{\alpha}\gamma$, and then
$$
\dot{V}(x,y)=-\frac{\alpha\gamma}{x}\left(x-\frac{\delta}{\alpha}\right)^{2}-\beta\left(\sigma-\frac{\delta}{\alpha}\gamma\right)y\leq0
$$ 
in the first quadrant, whenever $\gamma\delta\leq\alpha\sigma$. The now familiar argument shows that the interior of the first quadrant is a domain of attraction of $(\frac{\delta}{\alpha},0)$. Figure \ref{PredExt}(b) shows the vector field and the flow of the system after the bifurcation.  Recall that the $y$ coordinate stands for the number of predators, and it is $0$ at the attracting equilibrium. In biological terms, this means that the population of predators is driven to extinction. The growth and the death rates line up so that there is not enough prey to sustain it. This is also similar to the (heavily) damped harmonic oscillator, only there both variables are driven to $0$.

\section{Virus Dynamics}\label{S4}
\begin{figure}[!ht]
\begin{centering}
(a)\ \ \ \includegraphics[width=0.43\textwidth]{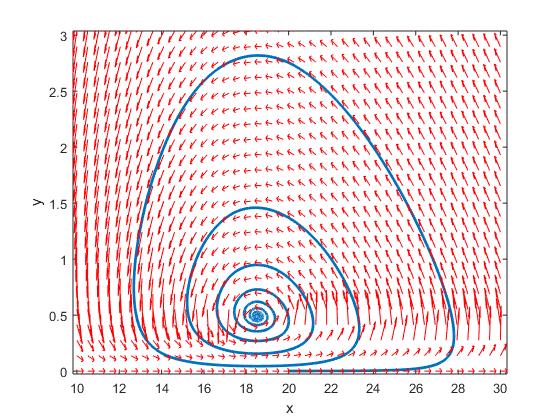} \hspace{0in} (b)\ \ \ \includegraphics[width=0.43\textwidth]{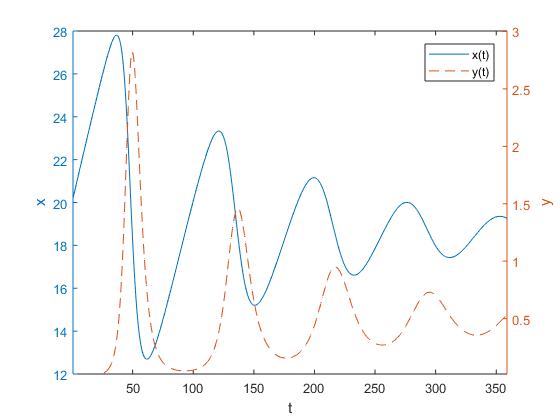}\par
\end{centering}
\caption{\label{VirFlow} (a) Vector field and a trajectory of the basic virus dynamics model ($R>1$);  (b) Time evolution of susceptible and virus producing cells.}
\end{figure}

So far we interpreted the variables of system \eqref{DampedPP} as the population numbers of either predators and prey or infected and susceptible individuals during epidemics. But, as with harmonic oscillators, there are many other interpretations. In this section we will interpret them as the numbers of virus producing and susceptible to infection cells in an organism, rather than individuals. It will also give us an opportunity to illustrate how Lyapunov functions can help with more than just establishing global attraction.

In virus dynamics $\delta$ is the creation rate of susceptible (healthy) cells, $\frac{\delta}{\alpha}$ is the stable density of cells in the absence of a virus, $\frac1{\sigma}$ is the average life span of an infected cell, and $\gamma$ determines the rate of infection. The ratio $R=\frac{\gamma\delta}{\alpha\sigma}$ is then the average number of cells infected by a single virus producing cell, and it is called the {\it basic reproductive ratio} \cite{BJM}. As follows from our analysis in previous sections, if $R\leq1$ the virus is driven to extinction, and if $R>1$ the number of virus producing cells eventually stabilizes at $\frac{a}{\beta}=\frac{\gamma}{\beta}\frac{\delta}{\sigma}(1-\frac1R)$, see Figure \ref{VirFlow}. In other words, $R$ is a {\it bifurcation parameter}, with the bifurcation at $R=1$.

But the basic reproductive ratio is also a measurable quantity taken as a measure of the viral load, the amount of virus present in the organism. In the original basic model of virus dynamics, it was assumed that $\beta=\gamma$, just as in the original Lotka-Volterra model, i.e. the number of susceptible cells decreased only due to infection (and natural dying off), and the number of infected cells increased by the same amount. However, this model predicted unrealistically low viral loads after virus inhibitor treatments that reduced $\gamma$. In one of the modifications proposed in \cite{BJM}, $\beta=\gamma+q$, where $q$ is the rate of virus-induced killing of susceptible cells, e.g. due to the immune response. This gives exactly the ``damped" predator-prey model.

Let us consider the following question: how much of an outbreak can we expect at the peak of infection after a small number of virus producing cells is introduced into a susceptible population? To answer it, look at the oval depicted in Figure \ref{TrapOval}\,(a). It can be characterized by the condition that its right tip is tangent to the line $x=\frac{\delta}{\alpha}$. Since this oval is a level set of $V$ from \eqref{Lyapunov1} the gradient of $V$ is perpendicular to the $y$-axis at the point of tangency. This gives $y=\frac{a}{\beta}$ for the $y$ coordinate of that point, which is equal to the $y$ coordinate of the stable equilibrium. Due to the direction of the field vectors, trajectories that pass to the left or under the oval must enter it since they can not cross the $x$-axis, the $y$-axis, or the line $x=\frac{\delta}{\alpha}$ outward. Once they enter it they can not leave, since $V$ is a Lyapunov function. This means that $y(t)\leq\overline{y}$, where $\overline{y}$ is the $y$-coordinate of the oval's top. Thus, $\overline{y}$ gives a desired estimate on the peak number of virus producing cells.
\begin{figure}[!ht]
\begin{centering}
(a)\ \ \ \includegraphics[width=0.43\textwidth]{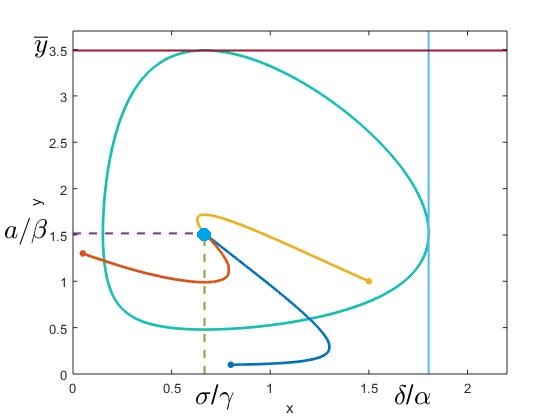} \hspace{0in} (b)\ \ \ \includegraphics[width=0.43\textwidth]{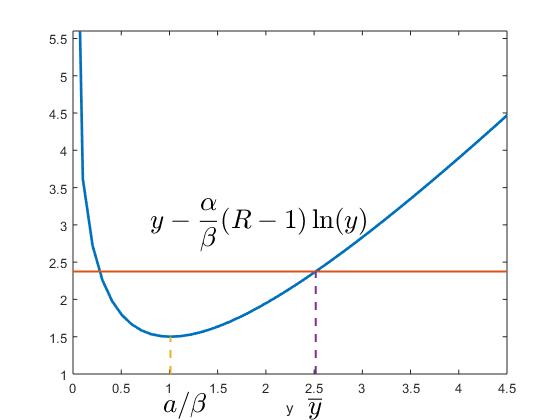}\par
\end{centering}
\caption{\label{TrapOval} (a) Trapping oval for the ``damped" predator-prey system; (b) Solving for the upper bound $\overline{y}$.}
\end{figure}

To find it, note that the oval's equation is $V(x,y)=C$, where $C$ can be found from the condition that the point $(\frac{\delta}{\alpha},\frac{a}{\beta})$ lies on it, i.e. $C=V(\frac{\delta}{\alpha},\frac{a}{\beta})$. Our $\overline{y}$ is the larger solution to $V(\frac{\sigma}{\gamma},y)=V(\frac{\delta}{\alpha},\frac{a}{\beta})$, see Figure \ref{TrapOval}\,(b). We can express this equation more explicitly in terms of $R$:
$$
y-\frac{\alpha}{\beta}(R-1)\ln y=\frac{\sigma}{\beta}(R-1-\ln R)+\frac{\alpha}{\beta}(R-1)\left(1-\ln\frac{\alpha}{\beta}(R-1)\right)\,.
$$
This is a transcendental equation for $\overline{y}$ that can not be solved analytically, but the solution can be easily found numerically for specific parameter values. Now consider the limiting case of large infection rates, when $\gamma\to\infty$ and $\frac{\gamma}{\beta}=\frac{\gamma}{\gamma+q}\to1$. Then $R\to\infty$ and $\frac{\alpha}{\beta}(R-1)\to\frac{\delta}{\sigma}$. The equation for $\overline{y}$ simplifies to 
$$
y-\frac{\delta}{\sigma}\ln y=\frac{\delta}{\alpha}+\frac{\delta}{\sigma}\left(1-\ln\frac{\delta}{\sigma}\right), 
$$
and gives a finite value. In other words, the maximal size of the outbreak remains bounded even for arbitrarily large infection rates!

\section{Plant Growth}\label{PlantGrowth}

Let us switch from viral infections to growth of plants. Bessonov and Volpert proposed a model of early shoot growth from the seed that involves water flow transport of a nutrient to the top of the shoot, where a growth hormone regulates the creation of new cells \cite{BV}. After neglecting diffusion and making simplifying assumptions to eliminate the convection equation the model can be reduced to a three-dimensional system:
\begin{equation}\label{BVmodel}
\begin{cases}\dot{x}=\frac{1}{L}(v-\gamma xy)\\
\dot{y}=\gamma xy-\sigma y\\
\dot{L}=f(y).
\end{cases}
\end{equation}
Here $x$ and $y$ are the nutrient and the hormone concentrations, $L$ is the length of the shoot, $v$ is the water flow speed, $\gamma$ is the rate of hormone production and $\sigma$ is the rate of its consumption. The $\frac{1}{L}$ factor accounts for the dilution of the nutrient over the longer columns of water in longer shoots. The growth function $f(y)$ is a non-negative monotone increasing threshold function, i.e. it introduces a threshold level $y_f$ the hormone concentration has to reach for the growth rate to be non-zero, Figure \ref{ThreshFun}(a). 
\begin{figure}[!ht]
\begin{centering}
(a)\ \ \ \includegraphics[width=0.43\textwidth]{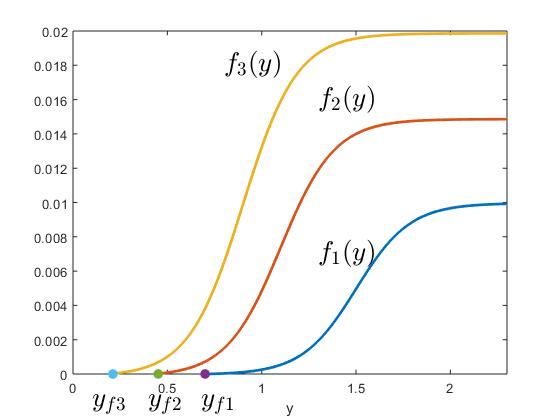} \hspace{0in} (b)\ \ \ \includegraphics[width=0.43\textwidth]{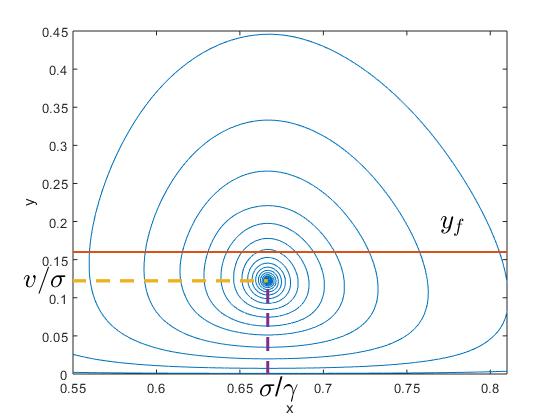}\par
\end{centering}
\caption{\label{ThreshFun} (a) Threshold functions for the plant growth model; (b) Phase trajectory of \eqref{BVmodel} projected to the $x$-$y$ plane.}
\end{figure}
If $y\leq y_f$ then $L$ is constant and the dynamics reduces to the first two equations. They are of the form \eqref{DampedPP} with $\alpha=0, \beta=\frac{\gamma}{L}$, and $\delta=\frac{v}{L}$. This reduced system has only one equilibrium in the first quadrant, $(\frac{\sigma}{\gamma},\frac{v}{\sigma})$. If $\frac{v}{\sigma}\leq y_f$  then $(\frac{\sigma}{\gamma},\frac{v}{\sigma}, L_*)$ is an equilibrium of the full system \eqref{BVmodel} with any fixed $L=L_*$, see Figure \ref{ThreshFun}(b). If not, the system has no equilibria. 
\begin{figure}[!ht]
\begin{centering}
(a)\ \ \ \includegraphics[width=0.43\textwidth]{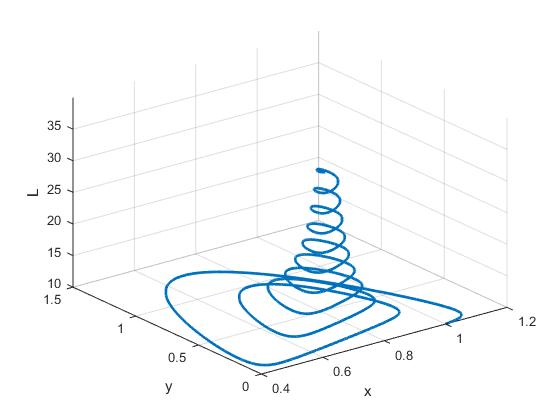} \hspace{0in} (b)\ \ \ \includegraphics[width=0.43\textwidth]{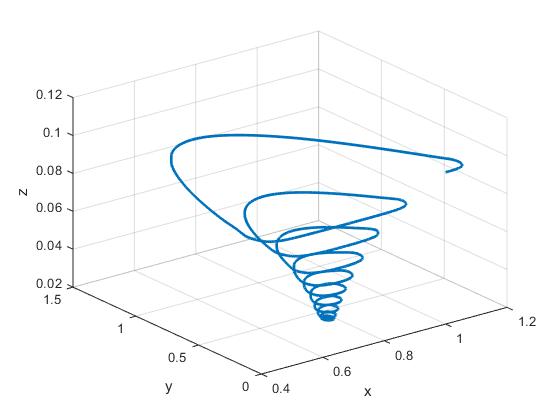}\par
\end{centering}
\caption{\label{3DPPlane} Phase trajectories of: (a) system \eqref{BVmodel};  (b)  system \eqref{BVzmodel}.}
\end{figure}
Since $L$ can grow without a bound it is convenient to use $z=\frac{1}{L}$ as a variable instead, which converts \eqref{BVmodel} into 
\begin{equation}\label{BVzmodel}
\begin{cases}\dot{x}=vz-\gamma zxy\\
\dot{y}=\gamma xy-\sigma y\\
\dot{z}=-f(y)z^2.
\end{cases}
\end{equation}
Now growth of $L$ to infinity is replaced by convergence of $z$ to $0$, Figure \ref{3DPPlane}.

In three dimensions dynamic behavior can be even more diverse than in the plane. In addition to escapes to infinity and limit cycles we can, in principle, encounter chaotic behavior with omega-limit set being a fractal. This is what happens in the famous example of Lorenz's strange attractor \cite{Via}, which also comes from simplifying convection equations.
As before, to narrow down the range of possibilities we look for a Lyapunov function, but it takes an additional technique to find the right ansatz. From previous sections, we know a Lyapunov function \eqref{Lyapunov1} for the first two equations when $z$ (and hence $L$) are fixed:
\begin{equation}\label{Lyapunov3D}
V(x,y,z):=\gamma x - \sigma\ln x + \gamma z\,y - \gamma z\,\frac{v}{\sigma}\ln y.
\end{equation}
But it may not be a Lyapunov function for the full system. Indeed, the derivative
$$
\dot{V}(x,y,z):=\nabla V\cdot(\dot{x},\dot{y},\dot{z})=-\frac{\gamma^2vz}{\sigma x}\left(x-\frac{\sigma}{\gamma}\right)^2-\gamma f(y)\left(y-\frac{v}{\sigma}\ln y\right)z^{2}\,$$
may not be non-positive in the entire first octant since the sign of $y-\frac{v}{\sigma}\ln y$ changes there. Fortunately, the last equation in \eqref{BVzmodel} and the shape of $f(y)$ imply that the function $W(x,y,z):=z$ is also a Lyapunov function! This just rephrases the fact that the length of the growing shoot never decreases. Since sums and positive multiples of non-positive derivatives are non-positive positive linear combinations of Lyapunov functions are good candidates for Lyapunov functions. Adding $m\gamma z$ to $V$, we obtain our new ansatz: $V_m(x,y,z):=V(x,y,z)+m\gamma z$, with $m$ to be determined. Its derivative along the trajectories is:
$$
\dot{V}_m(x,y,z)=-\frac{\gamma^2vz}{\sigma x}\left(x-\frac{\sigma}{\gamma}\right)^2-\gamma f(y)\left(y-\frac{v}{\sigma}\ln y+m\right)z^{2}\,.
$$
We can make it non-positive inside the entire first octant if we choose $m$ so that $y-\frac{v}{\sigma}\ln y+m\geq0$ for all $y>0$. This is possible because $y-\frac{v}{\sigma}\ln y$ has a global minimum for $y>0$. It is $\frac{v}{\sigma}(1-\ln\frac{v}{\sigma})$, and we can simply take any $m>\frac{v}{\sigma}(1-\ln\frac{v}{\sigma})$.

With this function, and a trapping region bound by its level set, the Lyapunov theorem tells us that the omega-limit set of a point inside the first octant is contained in 
$$
\dot{V}_m^{-1}(0)=\{(x,y,z)\,\big|\,z=0\textrm{ or }x=\frac{\sigma}{\gamma}, f(y)=0\}\,.
$$
Since $z(t)\geq0$ is monotone decreasing we also know that $z(t)\to z_*\geq0$. The case $z_*=0$ means that $L(t)\to\infty$, and this is the only option if $\frac{v}{\sigma}>y_f$, which is unbiological. But if  $\frac{v}{\sigma}<y_f$ and $z_*>0$ we can say more.
\begin{theorem}
Suppose $\frac{v}{\sigma}<y_f$ and $x_{0},y_{0}, L_0>0$. Then either $L(t)\to\infty$ (unbounded growth), or there is a stopping time $T_*>0$ and a final length $L_*$ such that $L(t)=L_*$ is constant for $t\geq T_*$.
In the latter case the nutrient and the hormone concentrations approach the equilibrium values $\frac{\sigma}{\gamma},\frac{v}{\sigma}$, respectively.
\end{theorem}
\begin{proof} If $z_*>0$ then $L(t)\leq L_*:=\frac1{z_*}$. Moreover, since the trajectories approach $(\frac{\sigma}{\gamma},\frac{v}{\sigma}, z_*)$, and $\frac{v}{\sigma}<y_f$, their $x$-$y$ projections must for large $t>0$ stay in a disk contained entirely under the threshold $y_f$. When this is so $L(t)$ is constant since $f(y)=0$. Therefore, there is the smallest time $T_*$ after which the growth stops, and then $L_*=L(T_*)$ is the final length.
After $T_*$ the dynamics of $x(t),y(t)$ is determined by the first two equations of \eqref{BVmodel} with $z=z_*$. Since $\gamma\delta-\alpha\sigma=\gamma vz_*>0$ Theorem \ref{GlobAttr} applies to them.
\end{proof}

Figure \ref{LXYplot} shows the growth pattern for different values of parameters, and typical behavior of the nutrient and the hormone concentrations. As one can see, the growth occurs in spurts interrupted by quiet periods. The system may spend a long time in a quiet period before resuming growth again. To make sure in simulations that the growth stops for good, we can use the Lyapunov function again. For every fixed value of $L$ (and $z$), the function $V_{m}(x,y,z)$ defines a family of Lotka-Volterra ovals in the $x$-$y$ plane. Consider the oval that touches the threshold line $y=y_{f}$ at the upper tip. Its equation can be determined from the fact that $(\frac{\sigma}{\gamma},y_{f})$ lies on it. If in the course of a simulation $V_{m}(x,y,z)\leq V_{m}(\frac{\sigma}{\gamma},y_{f},z)$, then growth can never resume, and $L_{*}=\frac{1}{z}$ is the final length ($m$ may be set to $0$ if $y_{m}<1$).
\begin{figure}[!ht]
\begin{centering}
(a)\ \ \ \includegraphics[width=0.43\textwidth]{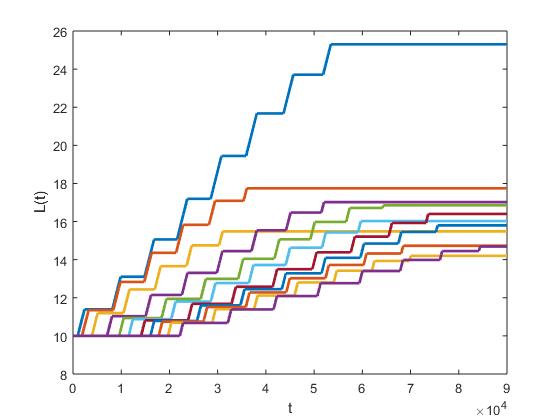} \hspace{0in} (b)\ \ \ \includegraphics[width=0.43\textwidth]{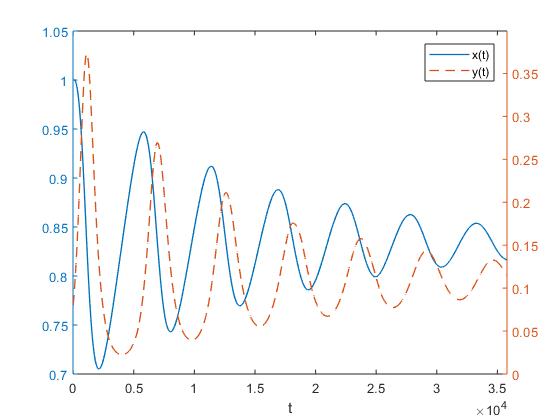}\par
\end{centering}
\caption{\label{LXYplot} Time evolution in \eqref{BVmodel}: (a) shoot length for different values of $\sigma$; (b) nutrient (solid) and hormone (dashed) concentrations.}
\end{figure}

This pattern was predicted also by the full Bessonov-Volpert model. Simulations show that for realistic initial values (small lengths and hormone concentrations) the growth does eventually stop. This raises many interesting problems, like finding explicit estimates for the final length and the stopping time in terms of the initial values. But answering them would probably require techniques beyond the use of Lyapunov functions.

\bigskip
\noindent{\large\bf Technology Used:} Simulations and figures were made in MATLAB. The \textit{pde45} function was used to solve systems of differential equations numerically. The \textit{plot} function was used for the 2D plots, the 3D plots were made with \textit{surf} (for surfaces) and \textit{plot3} (for trajectories). The vector fields in Figures \ref{SlopePP}(a), \ref{PredExt}(b), and \ref{VirFlow}(a) were plotted with the \textit{vectfieldn} package.

\bigskip
\noindent{\large\bf Acknowledgments:} The research for this paper 
was completed during the Research Experiences for Undergraduates (REU) program held at the University of Houston-Downtown in the summer of 2018. We would like to thank our fellow REU participants Dr.\,Youn-Sha Chan, Dr.\,Michael Tobin, Tomas Bryan, Roberto Hernandez and Michael Zhang for their helpful insights. This work was supported by NSF grant \#1560401.

\end{document}